\documentclass[letterpaper,10pt,reqno,final]{amsart}
\usepackage{amsmath}%
\usepackage{amsfonts}%
\usepackage{amssymb}%
\usepackage{graphicx}
\usepackage{functan}
\usepackage{mathrsfs}
\usepackage[all,cmtip]{xy}
\newtheorem{theorem}{Theorem}
\theoremstyle{plain}

\newtheorem{corollary}{Corollary}

\newtheorem{definition}{Definition}
\newtheorem{example}{Example}

\newtheorem{remark}{Remark}

\numberwithin{equation}{section}
\numberwithin{theorem}{section}
\numberwithin{corollary}{section}
\numberwithin{lemma}{section}
\numberwithin{example}{section}
\numberwithin{remark}{section}
\numberwithin{definition}{section}
\begin{document}
\title[Approximation of Contractive Semigroups]{On the Approximation of Contractive Semigroups of Operators in Discretizable Hilbert Spaces}
\author{Fredy Vides}
\address[F. Vides]
{Escuela de Matem{\'a}tica y Ciencias de la Computaci{\'o}n \newline%
\indent Universidad Nacional Aut{\'o}noma de Honduras}%
\email[F. Vides]{fvides@unah.edu.hn}%
\urladdr{http://fredyvides.6te.net}
\date{August 16, 2010}
\dedicatory{Dedicated to the memory of professor Salvador Llopis.}

\begin{abstract}
  The Computation of discrete Contractive semigroups becomes necessary when we deal with
  several types of evolution equations in Discretizable Hilbert spaces, in
  this work we study some properties of the discrete forms of the contractive
  semigroups induced by an approximation scheme in a prescribed Hilbert space, we 
  also deal with the implementation of computational methods in this
  Hilbert Space and apply some of the results presented here in the Heisenberg 
  representation of quantum dynamical semigroups. 
  \end{abstract}
\maketitle  

%\tableofcontents

\section{Introduction}

In this article we will work with evolution equations defined on a Hilbert
space $H : = H^n (G)$ that have the form:
\begin{equation}
  \left\{ \begin{array}{l}
    u' (t) = Au (t) + f (t)\\
    u (0) = u_0
  \end{array} \right. \label{first}
\end{equation}
with $A \in \mathcal{L}(H)$ constant in time and with $Re \: \sigma (A)
\leqslant 0$ and where $u_0 \in H$. In (\ref{first}) we have that the operator
$A$ represents in some suitable sense the spatial differential and boundary condition operators.

The set $\left\{ s_t : t \in \mathbb{T} \subseteq \mathbb{R} \right\}$ is
called a contractive semigroup generated by $A$ if we have that
$
  \lim_{h \rightarrow 0^+} h^{- 1} (s_h -\mathbf{1}) = A
$
and also

\begin{description}
	\item[S1]  $s_t s_v = s_{t + v}$, $t, s \in \mathbb{T}$ and $s_0
  =\mathbf{1}$;
  \item[S2] $\left\| s_t \right\|_{\mathcal{L}(H)} \leqslant 1$, $t \in
  \mathbb{T}$;
  \item[S3] $s_t x \in C (\mathbb{T}, H) \cap C^1 (\mathbb{T}, H)$, $t \in
  \mathbb{T}$.
\end{description}

In the following sections we will present the abstract setup needed to perform
the computation of the discrete semigroups of operators and the corresponding numerical
analysis of the behavior of its elements, also we will apply the results here presented
to the Heisenberg representation of quantum dynamical semigroups.

\section{Discretization Schemes}

\subsection{Discretizable Hilbert Spaces}

In general we will have that a discretizable Hilbert space $H (G)$, with $G \subset
\mathbb{R}^N$ compact, will be considered as any linear space that is both a
separable and a reproducing kernel Hilbert space. One of the basic elements
that we will use to perform the general discretization process is the grid
that will be defined as follows.

\begin{definition}
  Grid: For two given sets $G \subseteq \mathbb{R}^N$ and $\mathbb{G}=\{0,
  \cdots, M_m \} \subseteq \mathbb{Z}$, with $M_m$ a number that depends on a
  fixed number $m$, a fixed value $h \in \mathbb{R}^N$ and a bijection $f :
  \mathbb{Z} \rightarrow \mathbb{R}^N : \mathbb{G} \ni k \mapsto g \in
  \mathbb{R}^N$, the set $G_{m, h} =\{g_k \in G : g_k = f (k), k \in
  \mathbb{G}\}$ is called a grid in $G$ of size $h$ and length $M_m$ or
  simply a grid.
\end{definition}

For a given discretizable Hilbert space $H (G)$ one can define an operator
$P_{m, h} \in \mathcal{L} (H)$ called particular projector and defined in the
following way.

\begin{definition}
  Particular Projector: An operator $P_{m, h} \in \mathcal{L}(H)$, with $H$ a
  discretizable hilbert space, that satisfies the relations:
  \begin{eqnarray}
    P^2_{m, h} = P_{m, h} &  &  \label{proj1}\\
    P_{m, h} x \conv{}{h \to 0^+} x&  & \\
    \left\| P_{m, h} -\mathbf{1} \right\|_{\ast} \leqslant c_{\ast}
    h^{\mu_m} &  &  \label{proj2}
  \end{eqnarray}
  will be called a perticular projector, in (\ref{proj2}) $\left\| \cdot
  \right\|_{\ast}$ represents any prescribed norm in $\mathcal{L}(H)$ and
  $\mu_m$ is a number that depends on $m$ that will be called projection order
  with respect to $\left\| \cdot \right\|_{\ast}$.
\end{definition}

For a given particular projector $P_{m, h}$ in a discretizable Hilbert space
$H (G)$ we will denote by $H_{m, h}$ its corresponding subspace. A particular
projector $P_{m, h}$ can be factored in the form:
\begin{equation}
  P_{m, h} = p_{m, h} p_{m, h}^{\dagger} \label{pfactored}
\end{equation}
where the operators $p^{\dagger}_{m, h} \in \mathcal{L} (H, H^{\ast})$ and
$p_{m, h} \in \mathcal{L} (H^{\ast}, H)$ are called decomposition and
expansion factors of $P_{m, h}$ respectively. For each $m \in \mathbb{Z}^+$ a
particular projector $P_{m, h} \in \mathcal{L} (H)$ is related to a basis
$\mathscr{P} =\{p_1, \cdots, p_{N_m} \} \subseteq H$ through the following
expression:
\begin{equation}
  P_{m, h} p_k = p_k, \forall p_k \in \mathscr{P}
\end{equation}
also we will have that each $p_k\in\mathscr{P}$ will satisfy prescribed conditions $B p_k=P^b_k, x\in\partial G$ that are compatible with the boundary value conditions of the problem described by \eqref{first} in some suitable sense, and that the decomposition factor $p^{\dagger}_{m, h}$ in
(\ref{pfactored}) is determined by a prescribed grid $G_{m, h} \subseteq G$
through the relation:
\begin{equation}
  p^{\dagger}_{m, h} x = \hat{x} =\{c_k (x, G_{m, h})\}, k \in
  \mathbb{G}=\{1, \cdots, N_m \}.
\end{equation}

For a given discretizable Hilbert space $H$ whose inner product is induced by
the inner product map $\mathcal{M \in \mathcal{L}} (H, H^{\ast})$ in the following way
\begin{equation}
  \left\langle x, y \right\rangle_H := \mathcal{M} [x] (y)
\end{equation}
one can define a particular representation given by
\begin{equation}
  \mathcal{M}_{m, h} [p^{\dagger}_{m, h} \cdot] (p^{\dagger}_{m, h} \cdot)
  := \mathcal{M} [P_{m, h} \cdot] (P_{m, h} \cdot)
\end{equation}
that will recive the name of inner product matrix form relative to $H_{m, h}
:= P_{m, h} H$, it can be seen that
\begin{eqnarray*}
  \scalprod*{H_{m,h}}{x}{y}&=& \mathcal{M}_{m,h}[p^{\dagger}_{m,h}x](p^{\dagger}_{m,h}y) \\
  &=&[p^{\dagger}_{m,h}y]^{\ast}\mathcal{M}_{m,h}[p^{\dagger}_{m,h}x] \\
  &=& \mathcal{M}[P_{m,h}x](P_{m,h}y)\\
  & =& \left\langle P_{m, h} x, P_{m, h} y \right\rangle_H \\ 
\end{eqnarray*}
from this relations we can obtain the following results.

\begin{theorem}
  Every inner product matrix form is symmetric and positive definite (SPD).
\end{theorem}  
  \begin{proof}
    It can be seen that for a discretizable Hilbert space $H$ and a given
    particular projector $P_{m, h}$ in $H$, with basis $\mathscr{P} =\{p_1,
    \cdots, p_{N_m} \}$, we will have that
    \begin{eqnarray*}
      \left(\mathcal{M}_{m,h}\right)_{i,j}&=& \mathcal{M}_{m,h}[p^{\dagger}_{m,h}p_i](p^{\dagger}_{m,h}p_j)\\
      &=& \left\langle P_{m, h} p_i, P_{m, h} p_j \right\rangle_H  \\
      &=& \left\langle p_i, p_j \right\rangle_H  \\
      &=& \overline{\left\langle p_j, p_i \right\rangle}_H  \\
      &=& \overline{\left\langle P_{m, h} p_j, P_{m, h} p_i \right\rangle}_H \\
      &=& \overline{\mathcal{M}_{m, h} [p^{\dagger}_{m, h} p_j] (p^{\dagger}_{m,
      h} p_i)} \\
      &=& \overline{(\mathcal{M}_{m, h})}_{j, i} 
    \end{eqnarray*}
    and this implies that $\mathcal{M}_{m, h} = \mathcal{M}^{\ast}_{m, h}$.
    Now since
    \[ 0 \leqslant \left\| x \right\|^2_{H_{m, h}} = \left\langle x, x
       \right\rangle_{H_{m, h}} =\mathcal{M}_{m, h} [p^{\dagger}_{m, h} x]
       (p^{\dagger}_{m, h} x) \]
    we will have that $\mathcal{M}_{m, h} [x] (x) > 0$ for each $0 \neq x \in
    H \backslash Ker \: P_{m, h}$.
  \end{proof}

\begin{corollary}
Every inner product matrix form is invertible.
\end{corollary}

\subsection{Discretization of Operators}

Using particular projectors one can obtain for a given operator $B \in
\mathcal{L}(H)$ a corresponding representation defined by the following
definition.

\begin{definition}
  Particular Representation of an Operator. For a given operator $B \in
  \mathcal{L} (X, Y)$ being $X, Y$ discretizable Hilbert spaces and being
  $X_{m, h}, Y_{m, h}$ the subspaces relative to the particular projectors
  $P_{m, h} \in \mathcal{L} (X), Q_{m, h} \in \mathcal{L} (Y)$, the operator
  $B_{m, h} \in \mathcal{L} (X^{\ast}_{m, h}, Y^{\ast}_{m, h})$ given by
  \begin{equation}
    B_{m, h} := q^{\dagger}_{m, h} Bp_{m, h} \label{partop}
  \end{equation}
  will be called particular representation of $B$.
\end{definition}

Once we have computed the particular representation of a given operator $B \in
\mathcal{L} (X, Y)$ over a discretizable Hilbert spaces $X, Y$, in prescribed
subspaces $X_{m, h} \leqslant X, Y_{m, h} \leqslant Y$ determined by a
particular projectors $P_{m, h}, Q_{m, h}$, we will define the approximation
order of a particular representation as follows.

\begin{definition}
  Approximation order of a particular representation. We say that the
  particular representation $B_{m, h} \in \mathcal{L} (X^{\ast}_{m, h},
  Y^{\ast}_{m, h})$ of an operator $B \in \mathcal{L} (X, Y)$ is of order
  $\nu_m$ (with $\nu_m$ a value that depends of the prescribed number m) with
  respect to a given norm $\left\| \cdot \right\|_{\ast}$ in $Y$ if for each
  $x \in X$ there exists $c_{\ast}$ that does not depend on $h$ such that:
  \begin{equation}
    \left\| B_{m, h} p^{\dagger}_{m, h} x - q^{\dagger}_{m, h} Bx
    \right\|_{\ast} \leq c_{\ast} h^{\nu_m}
  \end{equation}
\end{definition}

\section{Sobolev Chains and Particular Factorization}
\subsection{Sobolev Chain}
If for a given Discretizable Hilbert space $X_0 := X^n (G)$ and a
prescribed sequence of operators in $\mathcal{L} (X_0)$ $\mathscr{B}:=\{b_k\}_{k = 0}^n$ we can define a sequence of
Hilbert spaces of the form $\mathscr{X}:=\{X_k \}_{k = 0}^n$, that satisfy the
relation:
\begin{equation}
  X_0:=b_0 X,\:X_{k + 1} := b_{k+1} X_k, 0 \leq k \leq n-1
\end{equation}
the Hilbert space
\begin{equation}
  Y_n := \bigoplus_{0 \leq k \leq n} X_k
\end{equation}
equiped with the inner product
\begin{eqnarray}
\label{sobolev}
  \scalprod*{Y_n}{x}{y}&=& \sum_{0\leq k\leq n} \scalprod*{X_k}{x_k}{y_k}\\
  &=& \sum_{0 \leq k \leq n} \left\langle B_k x_0, B_k y_0 \right\rangle_{X_k} = \left\langle x_0, y_0 \right\rangle_{Y_n} 
\end{eqnarray}
with $B_k$ defined by 
\begin{equation}
B_k:=\prod_{0\leq j\leq k}b_k
\end{equation}
the pair $\mathscr{X}, \mathscr{B}$ described above will be called a Sobolev chain based on
$X_0$ and generated by $\mathscr{B}$. 

\subsection{Particular Factorization of Operators}

Sobolev chains are particularly useful when we are working with operators over a prescribed
discretizable Hilbert space $X$, in this work a particularly important kind of chains will 
be those chains that permit us to write an operator $A\in\mathcal{L}(X)$ in the following way
\begin{equation}
  A = aa^{\dagger}
\end{equation}
in this cases we can easily obtain a particular factorization of $A$ using the Sobolev chain $\{X_0,X_1\},\{\mathbf{1},a^\dagger\}$ that will have the
form
\begin{eqnarray}
   \mathcal{A}_{m, h} [\cdot] (\cdot) &:=& [p^{\dagger}_{m, h}
  \cdot]^{\ast} \mathcal{A}_{m, h} [p^{\dagger}_{m, h} \cdot] 
  \label{factor1}\\
    & =&\mathcal{M}_{m, h} [a^{\dagger}_{m, h} \cdot] (a^{\dagger}_{m, h}
  \cdot) \\
    & =& [p^{\dagger}_{m, h} \cdot]^{\ast} (a^{\dagger}_{m, h})^{\ast}
  \mathcal{M}_{m, h} (a^{\dagger}_{m, h}) [p^{\dagger}_{m, h} \cdot] 
  \label{factor3}\\
    & =& \left\langle a^{\dagger}_{m, h} p^{\dagger}_{m, h} x,
  a^{\dagger}_{m, h} p^{\dagger}_{m, h} x \right\rangle_{X_0} \\
    & =& \left\langle P_{m, h} \cdot, P_{m, h} \cdot \right\rangle_{X_1}  
\end{eqnarray}
the expression pressented in \eqref{factor1} will be called particular pactorization of $A$, the graphic form of the particular factorization of $A\in\mathcal{L}(X)$ and its relation to $A\in\mathcal{L}(X)$ itself and to this particular type of Sobolev chain, can be expressed by the following diagram
\begin{eqnarray*}
\xymatrix{
X_{1,m,h}^\ast \ar[d]_{\mathcal{M}_{m,h}}  &X_{0,m,h}\ar[l]_{\left({a_{m,h}^\dagger}\right)^\ast}\ar[d]_{\mathcal{A}_{m,h}}\ar[dr]^{A_{m,h}}\ar[r]^{a_{m,h}^\dagger} & X_{1,mh} \ar[d]^{a_{m,h}}\\
X_{1,m,h}^\ast \ar[r]_{a_{m,h}^\dagger} &X_{2,m,h}^\ast\ar[r]_{\mathcal{M}_{m,h}^{-1}} &X_{2,m,h}}
\end{eqnarray*}
from \eqref{factor1} and \eqref{factor3} we obtain the following results.
\begin{theorem}\label{tfp}
  For a given Sobolev chain $\{X_0, X_1 \},\{\mathbf{1},a^\dagger\}$ the matrix $\mathcal{A}_{m, h}$
  is symmetric and positive definite.
\end{theorem}
\begin{proof}
    It can be seen that
    \begin{eqnarray*}
       \mathcal{A}^{\ast}_{m, h} &=& ((a^{\dagger}_{m, h})^{\ast}
      \mathcal{M}_{m, h} (a^{\dagger}_{m, h}))^{\ast}\\
      & = &  (a^{\dagger}_{m, h})^{\ast} \mathcal{M}^{\ast}_{m, h}
      (a^{\dagger}_{m, h})\\
      & = &  (a^{\dagger}_{m, h})^{\ast} \mathcal{M}_{m, h} (a^{\dagger}_{m,
      h})\\
      & = & \mathcal{A}_{m, h}
    \end{eqnarray*}
    also we may check that
    \begin{eqnarray}
      {}[p^{\dagger}_{m, h} x]^{\ast} \mathcal{A}_{m, h} [p^{\dagger}_{m, h}
      x] &=& \left\langle P_{m, h} x, P^{}_{m, h} x \right\rangle_{X_1}
      \nonumber\\
      &=& \left\| P_{m, h} x \right\|_{X_0} \geq 0  \label{positive}
    \end{eqnarray}
    \begin{flushleft}
      then for $0 \neq x \in X_0 \begin{array}{l}
        \backslash
      \end{array} (Ker \: P_{m, h} \cap Ker \: a^{\dagger}_{m, h})$ we
      will have that  $[p^{\dagger}_{m, h} x]^{\ast} \mathcal{A}_{m, h}
      [p^{\dagger}_{m, h} x] > 0$.\end{flushleft}
  \end{proof}
  
In this point we are going to present a very useful property of an operator that will be defined by.

\begin{definition}
For a given Hilbet space X, an operator $B\in\mathcal{L}(X)$ is said to be accretive if for any $x\in X$ 
we have that
\begin{equation}
Re\:\scalprod*{X}{Bx}{x}\geq 0
\end{equation}
\end{definition}

From T.\ref{tfp} we can obtain the following.
  
\begin{corollary}
For a given Sobolev chain $\{X_0,X_1 \},\{\mathbf{1},a^\dagger\}$ the particular representation $A_{m,
    h}=\mathcal{M}_{m,h}^{-1}\mathcal{A}_{m,h}$ of $A=aa^\dagger$ is accretive.
\end{corollary}

And using this corollary and T.\ref{T11A} it is not very difficult to see that.

\begin{corollary}
For a given Sobolev chain $\{X_0,X_1 \},\{\mathbf{1},a^\dagger\}$ the particular representation $A_{m,
    h}=\mathcal{M}_{m,h}^{-1}\mathcal{A}_{m,h}$ of $A=-aa^\dagger$ satisfies the condition
    $Re\:\sigma(A_{m,h})\leq 0$.
\end{corollary}

\section{Discretization of Semigroups}

As a part of the process of studying the numerical solution to \eqref{first},
we start with the discretization of the semigroups induced by $A \in
\mathcal{L} (H)$, with this in mind we obtain some results that will be
presented below.

If we denote by $\{s_t : t \in \mathbb{T}\}$ the semigroup generated by $A
\in \mathcal{L} (H)$, then the discrete representation of it will be denoted
by $\{ \tilde{s}_k : k \in \mathbb{K}\}$, with $\mathbb{K} \subseteq
\mathbb{Z}^+_0$, discrete semigroups mimic some of the properties of the
continuum ones in the following way:
\begin{description}
  \item[DS1] $\tilde{s}_k \tilde{s}_j = \tilde{s}_{k + j}, k, j \in \mathbb{K}$
  \ and \ $\tilde{s}_0 =\mathbf{1}$;
  
  \item[DS2] $\left\| \tilde{s}_k \right\|_{\mathcal{L}(H_{m, h})} \leq 1, k \in
  \mathbb{K}$.
\end{description}
as in the first section, we will have that the elements of the discrete
semigroup will be related to the particular representation $A_{m, h} \in
\mathcal{L} (H^{\ast}_{m, h})$ of $A \in \mathcal{L} (H)$ through the
expression
$
  \lim_{\tau \rightarrow 0^+} \tau^{- 1} ( \tilde{s}_1 -\mathbf{1}) = A.
$
\subsection{Polynomial time discretization of semigroups}

If we rewrite the equation \eqref{first} using the particular representation of its spatial part we obtain the following abstract semidiscrete initial value problem 
\begin{equation}
u_{m,h}'(t)=A_{m,h}u_{m,h}(t)+f_{m,h}(t)
\label{second}
\end{equation}
with initial condition $u_{m,h}(0)=\hat{u}_0=P_{m,h}u_0$, whose exact solution can be computed using the time continuous semigroup $\{\hat{g}_t: t\in\mathbb{T}\}$, with $\hat{g}_t:=e^{t A_{m,h}}$, in the following way:
\begin{equation}
u_{m,h}(t)=\hat{g}_t u_{m,h}(0)+\int_0^t\hat{g}_{t-s}f(s)ds
\end{equation}
if we can take an abstract Taylor polynomial of {\textit{n-th}} order around $t = 0$ of
the solution to \eqref{first} when $f (t) = 0$, we obtain
\begin{equation}
  u_{\tau} := \sum_{k = 0}^n \frac{1}{k!} (\tau A)^k u_0 =
  \tilde{g}_{\tau} u_0 \label{polysem}
\end{equation}
here $\tilde{g}_{\tau}$ is called basic element of the discrete semigroup of
\textit{n-th} order relative to $A \in \mathcal{L} (A)$, because of the
following relation
\begin{equation}
  \text{$\{ \tilde{s}_{n, k} : \tilde{s}_{n, k} := \tilde{g}^k_{\tau}, k
  \in \mathbb{K}\}$}
\end{equation}
it is not very difficult to see that for a given polynomial integration scheme
that is exact for \textit{n-th} order polynomials described by
\begin{equation}
  \mathcal{Q}_n v := \sum_{j = 0}^n w_j v (j \tau)
\end{equation}
one can obtain a better approximation $\hat{s}_{n \tau} u_0$ to the solution
of \eqref{first} with respect to a given initial estimation in the following way
\begin{equation}
  \hat{s}_{n \tau} 
  \hat{u}_0 = \left[ \mathcal{Q}_n A \tilde{s}_{n, \cdot} \right]
  \hat{u}_0 +\mathcal{Q}_n \tilde{s}_{n, n - \cdot} f(\cdot) .
\label{better}
\end{equation}
\subsection{Semigroups generated by particular representations}

When we have a discrete semigroup $\{ \tilde{s}_{n, k} : \tilde{s}_{n,
k} := \tilde{g}^k_{\tau}, k \in \mathbb{K}\}$ where $\tilde{g}_{\tau}$
is defined in the same way as in \eqref{polysem}, and we also have that
$
  \lim_{\tau \rightarrow 0^+} \tau^{- 1} ( \tilde{s}_{n, 1} -\mathbf{1}) =
  A_{m, h}
$
with $A_{m, h} \in \mathcal{L} (X^{\ast}_{m, h})$ being the particular
representation of an operator $A \in \mathcal{L} (X)$, with $X$ a
discretizable Hilbert space, we say that the discrete semigroup is generated
by $A_{m, h}$, the polynomial that represents $\tilde{g}_{\tau}$ in this case
will described by
\begin{equation}
  \tilde{g}_{\tau} := \sum_{k = 0}^n \frac{1}{k!} (\tau A_{m, h})^k .
  \label{dtsemig}
\end{equation}

\subsection{Stability and Convergence}

In this section we will consider that for $m,h$ fixed a given particular representation $A_{m, h} \in
\mathbb{C}^{N_m \times N_n}$, of a prescribed accretive operator $A \in
\mathcal{L} (X)$ is related to a particular projector $P_{m,h}$, with basis $\mathscr{P}=\{ p_k \}_{k=0}^{N_m}$ in a discretizable Hilbert space $X$, and also that \eqref{dtsemig} can be expressed in the form
\begin{eqnarray}
   \tilde{g}_{\tau} &=& \sum_{k = 0}^n \frac{1}{k!} (P^{- 1}_A \tau D_A
  P_A)^k \\
  & = &  \sum_{k = 0}^n \frac{1}{k!} P^{- 1}_A \tau^k D^k_A P_A 
\end{eqnarray}
where $D_A \in \mathbb{C}^{N_m \times N_m}$ is a diagonal matrix defined by
\begin{equation}
  D_A :=diag\{\lambda_j\}= \left(\begin{array}{ccccc}
    \lambda_0 & 0 & \cdots & 0 & 0\\
    0 & \lambda_1 &  &  & 0\\
    \vdots &  & \ddots &  & \vdots\\
    0 &  &  & \lambda_{N_m-1} & 0\\
    0 & 0 & \cdots & 0 & \lambda_{N_m}
  \end{array}\right)
\end{equation}
with $\lambda_j \in \sigma (A_{m, h}), j \in \{1, \cdots, N_m \}$, and where
$P_A \in \mathbb{C}^{N_m \times N_m}$ is defined by
\begin{equation}
  P_A : = \left(\begin{array}{cccc}
    v_1 & v_2 & \cdots & v_{N_m}
  \end{array}\right)
\end{equation}
with $A_{m, h} v_j = \lambda_j v_j, j \in \{0, \cdots, N_m \}$, wich means
that the \textit{j-th} column of $P_A$ is the eigenvector that corresponds to
the \textit{j-th} eigenvalue of $A_{m, h}$. If $A \in \mathcal{L} (X)$ can be
factored using a Sobolev chain of the form $\{X_k \}_{k=0,1}, a$, with $X_0 = X$,
then from \eqref{factor3} we will have that
\begin{equation}
  A_{m, h} = (a^{\dagger}_{m, h})^{\ast} \mathcal{M}_{m, h} (a^{\dagger}_{m,
  h}) \label{factor37}
\end{equation}
and this implies that
\begin{equation}
  \tilde{g}_{\tau} = \sum_{k = 0}^n \frac{1}{k!} P^{\ast}_A \tau^k D^k_A P_A
  \label{dtsym}
\end{equation}
from T.\ref{T11A}, T.\ref{T15A} in appendix A and T.\ref{T16B} in
appendix B we get the following results concerning to stability of the approximation schemes.
\begin{theorem}
\label{ts1}
  Stability 1. If an operator $A \in \mathcal{L} (X)$ can be
  factored using a Sobolev chain of the form $\{X_0, X_1 \}, \{\mathbf{1},a^\dagger\}$ with $X_0 =
  X$ a discretizable Hilbert space, then the basic element $\tilde{g}_{\tau}$ of the discrete semigroup generated by $-A\in\mathcal{L}(X)$ and described in (\ref{dtsym}) will
  satisfy the relation
  \begin{equation}
    \left\| \tilde{g}_{\tau} \right\|_{\mathcal{L} (X_{m, h})} = \left\|
    \mathcal{M}^{1 / 2}_{m, h} \tilde{g}_{\tau} \mathcal{M}^{- 1 / 2}_{m, h}
    \right\|_2 \leq 1
  \end{equation}
  when $\tau = \alpha h^d /K_A$, with $\norm*{\infty}{\hat{A}_{m,h}}\leq K_A h^{-d}$, $\hat{A}_{m, h} =\mathcal{M}^{1 / 2}_{m,
  h} A_{m, h} \mathcal{M}^{- 1 / 2}_{m, h}$ and $0 < \alpha \leq 1$.
  \end{theorem}
   \begin{proof}
    From corollary C.\ref{T2inf} we will have that $\norm*{2}{\hat{A}_{m, h}} \leq  \norm*{\infty}{\hat{A}_{m, h}}$ and clearly $\tau = \alpha h^d /K_A \leq \norm*{\infty}{\hat{A}_{m, h}}^{-1}  \leq \norm*{2}{\hat{A}_{m, h}}^{-1}$. Now since $\tilde{g}_{\tau} = \hat{P}^{\ast}_A
    p_n (\tau \hat{D}_A) \hat{P}_A$ with
    \begin{equation}
      p_n (z) := \sum_{k = 0}^n \frac{1}{k!} z^k
    \end{equation}
    and if we represent by $\mu \in \mathbb{R}$ and $\kappa\in\mathbb{R}$ the values $\mu := \sup
    \{|\lambda|:\lambda \in \sigma ( \hat{A}_{m, h})\}$ and $\kappa:=\inf\{|\lambda|:\lambda\in\sigma(\hat{A}_{m,h})\}$ we will have that
    \begin{eqnarray*}
       \left\| \hat{P}^{\ast}_A p_m (\tau \hat{D}_A) \hat{P}_A \right\|_2 &=&
      \left\| \hat{P}^{\ast}_A p_m (\alpha h^d /K_A \hat{D}_A) \hat{P}_A \right\|_2  \\
      & \leq &\left\| \hat{P}^{\ast}_A p_m (\mu^{-1} \hat{D}_A) \hat{P}_A \right\|_2 \\
      &\leq& p_n\left(\alpha \frac{\kappa}{\mu}\right)\\
      &\leq& 1
    \end{eqnarray*}
  \end{proof}
  
  From the last result we can easily obtain the following
  
  \begin{corollary}
  Stability 2. If an operator $A \in \mathcal{L} (X)$ can be
  factored using a Sobolev chain of the form $\{X_0, X_1 \}, \{\mathbf{1},a^\dagger\}$ with $X_0 =
  X$ a discretizable Hilbert space, then the basic element $\tilde{g}_{\tau}$ of the discrete semigroup generated by $-A\in\mathcal{L}(X)$ and described in (\ref{dtsym}) will
  satisfy the relation
  \begin{equation}
  \norm*{\mathcal{L}(X_{m,h})}{\tilde{g}^k}\leq 1, \: k \geq 0.
  \end{equation}
  \end{corollary}
  
  Also for any given $k\geq 0$ it can be seen that
  
  \begin{theorem}
Cauchy condition. If $k\geq 0$ then $\norm*{X_{m,h}}{\tilde{u}((k+1)\tau)-\tilde{u}(k\tau)}\leq\tilde{c}_1v^k$ where $v$ is a value $\leq 1$.
\end{theorem}
\begin{proof}
Since $\tilde{u}(j\tau)=\tilde{s}_{n,j}\hat{u}_0$ we will have that
\begin{eqnarray*}
\norm*{X_{m,h}}{\tilde{u}((k+j)\tau)-\tilde{u}(k\tau)}&=&\norm*{X_{m,h}}{\left(\tilde{s}_{n,k+j}-\tilde{s}_{n,k}\right)\hat{u}_0}\\
&=&\norm*{X_{m,h}}{\tilde{s}_{n,k}\left(\tilde{s}_{n,j}-\mathbf{1}\right)\hat{u}_0}\\
&\leq&\norm*{\mathcal{L}(X_{m,h})}{\tilde{g}_\tau^k}\norm*{\mathcal{L}(X_{m,h})}{\tilde{g}_\tau^j-\mathbf{1}}\norm*{X_{m,h}}{\hat{u}_0}\\
&\leq&\norm*{\mathcal{L}(X_{m,h})}{p(\tau A_{m,h})}^k\norm*{\mathcal{L}(X_{m,h})}{q(\tau A_{m,h})}\norm*{X_{m,h}}{\hat{u}_0}\\
&\leq&\left(p\left(\frac{\kappa}{\mu}\right)\right)^k \left|q\left(1\right)\right|\norm*{X_{m,h}}{\hat{u}_0}\\
\end{eqnarray*}
with $p(z):=\sum_{k=0}^n\frac{1}{k!}z^k$, $q(z):=(p(z))^j-1$, $\mu:=\sup\{|\lambda|:\lambda\in\sigma(A_{m,h})\}$ and $\kappa:=\inf\{|\lambda|:\lambda\in\sigma(A_{m,h})\}$, taking $v=p(\kappa/\mu)$ and $\tilde{c}_1=|q(1)|\norm*{X_{m,h}}{\hat{u}_0}$ concludes the proof.
\end{proof} 

Concerning to convergence of the approximation schemes we can obtain the following result.
    
\begin{theorem}
\label{tc2}
  Convergence. If for a given accretive operator $B\in\mathcal{L}(X)$ and each $x\in X$ we have that $e^{k\tau A_{m,h}}$, $k\tau\in\mathbb{T}$, with $A=-B$, has approximation order $\nu_m$ with respect $\norm*{X}{\cdot}$ and if there exists $K_A\in\mathbb{R}_0^+$ such that we can take $\tau:=\alpha h^d /K_A$ with $K_A h^{-d} \geq\norm*{\infty}{A_{m,h}}$ and with $1\geq\alpha:= h^{\frac{\nu_m}{n+1}}$, where $n \in \mathbb{Z}^+_0$ is a prescribed
  number, then we will have that there exists a constant $C_2$ such that
  \[ \left\| p^{\dagger}_{m, h} u (k \tau) - \tilde{s}_{n, k}
     p^{\dagger}_{m, h} u ( 0) \right\|_{X} \leq C_2 h^{\nu_m} . \]
\end{theorem}
\begin{proof}
Here we will consider that $e^{k\tau A_{m,h}}$ has approximation order $\nu_m$ with respect to $\norm*{X}{\cdot}$ wich implies
\begin{eqnarray*}
\norm*{X_{m,h}}{\hat{u}(k\tau)-\tilde{s}_{n,k}\hat{u}_0}&\leq&\norm*{X_{m,h}}{\hat{u}(k\tau)-e^{k\tau A_{m,h}}\hat{u}_0}+\\&&\norm*{X_{m,h}}{e^{k\tau A_{m,h}}\hat{u}_0-\tilde{s}_{n,k}\hat{u}_0} \\
&\leq&c_2 h^{\nu_m}+\norm*{X_{m,h}}{\sum_{j=n+1}^\infty \frac{r_j}{j!}(k\tau A_{m,h})^j\hat{u}_0}\\
&\leq&c_2 h^{\nu_m}+\norm*{\mathcal{L}(X_{m,h})}{\frac{r_{n+1}}{(n+1)!}k^{n+1}\tau^{n+1} A_{m,h}^{n+1}}\norm*{X_{m,h}}{\hat{u}_0}\\
&\leq& 
c_2h^{\nu_m}+\\ 
&&
\frac{r_{n+1}}{(n+1)!}h^{\nu_m}k^{n+1}\norm*{\mathcal{L}(X_{m,h})}{\left(\frac{h^d}{K_A}A_{m,h}\right)^{n+1}}\norm*{X_{m,h}}{\hat{u}_0}
\\
&\leq&c_2 h^{\nu_m}+\frac{r_{n+1}}{(n+1)!}k^{n+1}h^{\nu_m}\norm*{X_{m,h}}{\hat{u}_0}\\
&=&\left(c_2+\frac{r_{n+1}}{(n+1)!}k^{n+1}\norm*{X_{m,h}}{\hat{u}_0}\right)h^{\nu_m}
\end{eqnarray*}
here $r_{k}:=1-k!C (\{b_j\},k)$ where $C (\{b_j\},k)$ are the multinomial coeficients that correspond to the coeficients $\{b_j\}$ of the abstract polynomial $\tilde{s}_{n,k}$, taking $C_2=c_2+\frac{k^{n+1}r_{n+1}}{(n+1)!}\norm*{X_{m,h}}{\hat{u}_0}$ concludes the proof.
\end{proof}

\section{Application to Evolution of Operators in the Heisenberg Picture}
In this section we will describe a basic procedure of implementation of the results presented in this work in the computation of evolution of observables of a quantum system in the Heisenberg picture, here we will consider that all the quantum systems are modeled in a discretizable Hilbert space $X$ with inner product $\scalprod*{X}{u}{v}$ given by
\begin{equation}
\scalprod*{X(G)}{u}{v}:=\int_{G}u\overline{v} d\mu(G).
\label{inex}
\end{equation}
where $d\mu(G)$ is the volume mesure element in $G$, also we will consider that we can take a particular projector $P_{m,h}\in\mathcal{L}(X,X_{m,h})$ compatible prescribed boundary value condtions in some suitable sense and whose decomposition and expansion factors are related to a prescribed grid $G_{m,h}\subset G$ and basis $\mathscr{P}:=\{p_k\}$ respectively.

\subsection{Quantum Dynamical Semigroups} 

For a given quantum system on a discretizable Hilbert space $X:=X^2(G)$ whose wave function $\psi \in C([0,T],X^2(G))\cap C^1([0,T],X^2(G))$ is modeled by a Schr{\"o}dinger equation of the form
\begin{equation}
E\psi(t)=H\psi(t)
\label{third}
\end{equation}
here $\mathcal{L}(X^n(G))\ni H:=pp^\dagger$, $p^\dagger\longrightarrow -\frac{i}{\hbar}\nabla+\beta$, with $\nabla:=\hat{e}_k\sum_k\partial_k$ and $E\longrightarrow\frac{i}{\hbar}\partial_t$. If \eqref{third} has initial value $\psi(0)=\psi_0\in X$ and is subject to boundary value conditions of the form $B_H\psi=\psi_b, x\in\partial G$.

If we take a scale where $\hbar=1$, we can obtain a particular representation of $H$ denoted by  $H_{m,h}$, using this representation \eqref{third} will take the form
\begin{equation}
\left\{
\begin{array}{l}
\hat{\psi}'(t)=-i H_{m,h} \hat{\psi}(t)\\
\hat{\psi}(0)=\hat{\psi}_0 
\end{array}
\right.
\label{eb4}
\end{equation}
the \textit{n-th} order semigroup $\{U_{n,k}: U_{n,k}:=G_{\tau}^k, k\in\mathbb{K}\}$ generated by $-iH_{m,h}$ will be called \textit{n-th} order quantum dynamical semigroup in the Schr{\"o}dinger representation where $G_\tau:=\sum_{k=0}^n\frac{1}{k!}(-i\tau H_{m,h})$, using the elements of this semigroup we can write the solution to \eqref{third} in the form
\begin{equation}
\hat{\psi}(k\tau)=U_{n,k}\hat{\psi}_0.
\end{equation}

When we work with particular representations of equations like \eqref{eb4} we will have as an application of T.\ref{ts1} that:

\begin{theorem}\label{tscs} Stability of Complex Semigroups. If an operator $A \in \mathcal{L} (X)$ can be
  particularly factored using a Sobolev chain of the form $\{X_0, X_1 \}, \{\mathbf{1},a^\dagger\}$ with $X_0 =
  X$ a discretizable Hilbert space, then the basic element of the $\tilde{G}_{\tau}$ of the discrete semigroup generated by $-A\in\mathcal{L}(X)$ and described by 
  \begin{equation}
  \tilde{G}_\tau:=\sum_{k=0}^{n}\frac{1}{k!}(i\tau A_{m,h}), \: \mathbb{Z}_0^+\ni n\geq 1
  \end{equation}
  will satisfy the relation
  \begin{equation}
    \left\| \tilde{G}_{\tau} \right\|_{\mathcal{L} (X_{m, h})} = \left\|
    \mathcal{M}^{1 / 2}_{m, h} \tilde{G}_{\tau} \mathcal{M}^{- 1 / 2}_{m, h}
    \right\|_2 \leq 1
  \end{equation}
  when $\tau = \alpha h^d /K_A$, with $\norm*{\infty}{\hat{A}_{m,h}}\leq K_A h^{-d}$, $\hat{A}_{m, h} =\mathcal{M}^{1 / 2}_{m,
  h} A_{m, h} \mathcal{M}^{- 1 / 2}_{m, h}$ and $0 < \alpha \leq 1$.
  \end{theorem}
   \begin{proof}
    Since $A_{m, h}$ is accretive we will have that $\hat{A}_{m, h}$ will
    be accretive too and from
    C.\ref{T2inf} we will have that $\norm*{2}{\hat{A}_{m, h}} \leq  \norm*{\infty}{\hat{A}_{m, h}}$ and clearly $\tau = \alpha h^d /K_A \leq \norm*{\infty}{\hat{A}_{m, h}}^{-1}  \leq \norm*{2}{\hat{A}_{m, h}}^{-1}$. Now since $\tilde{G}_{\tau} = \hat{P}^{\ast}_A
    p_m (i\tau \hat{D}_A) \hat{P}_A$ with
    \begin{equation}
      p_{n} (z) := \sum_{k = 0}^{n} \frac{1}{k!} z^k
    \end{equation}
    and if we represent by $\mu \in \mathbb{R}$ and $\kappa\in\mathbb{R}$ the values $\mu := \sup
    \{|\lambda|:\lambda \in \sigma ( \hat{A}_{m, h})\}$ and $\kappa:=\inf\{|\lambda|:\lambda\in\sigma(\hat{A}_{m,h})\}$ we will have that
    \begin{eqnarray*}
       \left\| \hat{P}^{\ast}_A p_m (i\tau \hat{D}_A) \hat{P}_A \right\|_2 &=&
      \left\| \hat{P}^{\ast}_A p_m (i\alpha h^d /K_A \hat{D}_A) \hat{P}_A \right\|_2  \\
      & \leq &\left\| \hat{P}^{\ast}_A p_m (i\mu^{-1} \hat{D}_A) \hat{P}_A \right\|_2 
      \\
      & \leq& \left|p_{n}\left(i\alpha\frac{\kappa}{\mu}\right)\right|\\
      &\leq 1&. 
    \end{eqnarray*}
  \end{proof}
  
Following a similar procedure to the followed in the prooves of T.\ref{tc2} and T.\ref{tscs} we can obtain the following result. 
  
  \begin{theorem}\label{tccs}
  Convergence of complex semigroups. If for a given observable (symmetric operator) $A\in\mathcal{L}(X)$ and for each $x\in X$ we have that $e^{i k\tau A_{m,h}}$, $k\tau\in\mathbb{T}$ has approximation order $\nu_m$ with respect $\norm*{X}{\cdot}$ and if there exists $K_A\in\mathbb{R}_0^+$ such that we can take $\tau:=\alpha h^d /K_A$ with $K_A h^{-d} \geq\norm*{\infty}{A_{m,h}}$ and with $1\geq\alpha:= h^{\frac{\nu_m}{n+1}}$, where $n \in \mathbb{Z}^+_0$ is a prescribed
  number, then we will have that there exists a constant $C_3$ such that
  \[ \left\| p^{\dagger}_{m, h} u (k \tau) - \tilde{s}_{n, k}
     p^{\dagger}_{m, h} u ( 0) \right\|_{X} \leq C_3 h^{\nu_m}  \]
where $\tilde{s}_{n,k}:=\tilde{G}_\tau^k$ with $\tilde{G}_\tau:=\sum_{k=0}^{n}\frac{1}{k!}(i\tau A_{m,h})^k$.
\end{theorem}
\begin{proof}
Here we will consider that $e^{ik\tau A_{m,h}}$ has approximation order $\nu_m$ with respect to $\norm*{X}{\cdot}$ and that the value $\mu_{m,h}:=\sup\{\|\lambda\|:\lambda\in\sigma(\mathcal{M}_{m,h}^{1/2}A_{m,h}\mathcal{M}_{m,h}^{-1/2})\}$ wich implies
\begin{eqnarray*}
\norm*{X_{m,h}}{\hat{u}(k\tau)-\tilde{s}_{n,k}\hat{u}_0}&\leq&\norm*{X_{m,h}}{\hat{u}(k\tau)-e^{i k\tau A_{m,h}}\hat{u}_0}+\\&&\norm*{X_{m,h}}{e^{i k\tau A_{m,h}}\hat{u}_0-\tilde{s}_{n,k}\hat{u}_0} \\
&\leq&c_3 h^{\nu_m}+\norm*{X_{m,h}}{\sum_{j=n+1}^\infty \frac{r_j}{j!}(i k\tau A_{m,h})^j\hat{u}_0}\\
&\leq&c_3 h^{\nu_m}+k^{n+1}h^{\nu_m}\left|Q_{n+1}\left(i\frac{h^d}{K_A}\mu_{m,h}\right)\right|\norm*{X_{m,h}}{\hat{u}_0}\\
&\leq&c_3 h^{\nu_m}+\frac{r_{n+1}}{(n+1)!}k^{n+1}h^{\nu_m}\norm*{X_{m,h}}{\hat{u}_0}\\
&=&\left(c_3+\frac{r_{n+1}}{(n+1)!}k^{n+1}\norm*{X_{m,h}}{\hat{u}_0}\right)h^{\nu_m}
\end{eqnarray*}
here $Q_{n+1}(z)$ is defined by
\begin{equation}
Q_{n+1}(z):=\frac{r_{n+1}}{(n+1)!}z^{n+1}\left(\sum_{k=0}^\infty q_{k}z^k\right)^{1/2}
\end{equation}
with $r_{n+1}:=1-(n+1)!C(\{a_{j}\},n+1)$, $q_0:=1$, $q_{j}:=\frac{1-j!C (\{a_{j}\},n+1)}{r_{n+1}},\:j\geq 1$ and where $C (\{a_{j}\},n+1)$ are the multinomial coeficients corresponding to the coeficients $\{a_j\}$ of the abstract polynomial $\tilde{s}_{n,k}$, taking $C_3=c_3+\frac{k^{n+1}r_{n+1}}{(n+1)!}\norm*{X_{m,h}}{\hat{u}_0}$ concludes the proof.
\end{proof}

Now, for any given observable $B\in\mathcal{L}(X)$ with particular representation given by $B_{m,h}\in\mathcal{L}(X_{m,h}^\ast)$ \textit{(operator)} we can obtain its Heisenberg evolution through the computation
\begin{equation}
\hat{B}_{n,k}:=U^\dagger_{m,k}B_{m,h} U_{n,k}
\end{equation}
the set $\{\hat{B}_{n,j}:\hat{B}_{n,j}=U_{n,j}^\dagger B U_{n,k} \}$ will be called quantum dynamical semigroup in the Heisenberg representation, also we can compute its expected value that will be described by
\begin{equation}
\mathbb{E}(\hat{B}_{n,k}):=\left\langle \hat{B}_{n,k}\right\rangle=\norm*{X_{m,h}}{U_{n,k}\hat{\psi}_0}^{-2}\scalprod*{X_{m,h}}{B_t \hat{\psi}_0}{\hat{\psi}_0}.
\end{equation}

\subsection{Computation of Quantum Dynamical Semigroups in the Heisenberg representation}

In this section we will present an example of numerical computation of a quantum dynamical semigroup in the Heisenberg representation and more specificly the Heisenberg evolution of the position operator $\mathbf{X}\in\mathcal{L}(X)$ a prescribed quantum system.

\begin{example}
For the quantum system consisting of a particle in a bidimensional box represented by $G=[-1,1]^2$ and whose wave function is modeled by the Schr{\"o}dinger equation
\begin{equation}
\label{ex1}
\left\{
\begin{array}{l}
\partial_t \psi = -i H \psi , \: x\in G \\
\psi(0)=\psi_0 \\
\psi=0, x\in\partial G
\end{array}
\right.
\end{equation}
here $H:=pp^\dagger$ with $p^\dagger=-i[\partial_x,\partial_y]$, for this example we will use the Sobolev chain $\{X_0,X_1\},\{\mathbf{1},p^\dagger\}$, also we will have that 
\begin{equation}
\scalprod*{X}{u}{v}:=\int_G u\overline{v} dxdy 
\end{equation}
and $G_{m,h}:=\{x_{j_k}\}\times\{x_{j_k}\}$, where $\{x_j\}$ the Gauss-Lobatto grid of order $2$, $\mathscr{B}:=\{\ell_{i_k}\otimes\ell_{j_k}\}$, $\ell_{i_k}$ the \textit{$i_k$-th} cardinal basis \textit{(Lagrange interpolating system)} with respect to $\{x_{j_k}\}$, wich means that $\ell_{j_k}(x_{i_k})=\delta_{j_k,i_k}$, with $\delta_{i,j}$ the Kronecker delta, using the particular factorization of $-i H_{2,1/8}$ we can obtain the particular decomposition of the spatial part of \eqref{ex1} in the following way
\begin{equation}
\hat{\psi}'(t)=-iH_{2,1/8}\hat{\psi}(t)
\end{equation}
with $H_{2,1/8}=\mathcal{M}_{2,1/8}^{-1}[p^\dagger_{2,1/8}]^\ast\mathcal{M}_{2,1/8}[p^\dagger_{2,1/8}]$, in this particular case we have that:
\begin{equation}
\mathcal{M}_{2,1/8}:=\left(\begin{array}{cc}
1	& 0\\
0	& 1
\end{array}\right)\otimes \mathcal{W}_{2,1/8}(w)
\end{equation}
with $\mathcal{W}_{2,1/8}(w):=diag\{w_{j_k}\}$ an operator that depends in a suitable sense on the basic integrating matrix $w\in\mathbb{R}^{3\times 3}$ of second order defined by 
\begin{equation}
w:=
\left(
\begin{array}{ccc}
	1/3& 0& 0\\
	0&4/3&0\\
	0&0&1/3
\end{array}
\right)
\end{equation}
and also we will have that
\begin{equation}
p_{2,1/8}^\dagger:=i\left(\begin{array}{cc}
1&0\\
0&1\\
\end{array}\right)\otimes
\mathcal{W}_{2,1/8}^{-1/2}(w)
\left[\begin{array}{c}
 D_{2,1/8}(d) \otimes \mathbf{1} \\
\mathbf{1} \otimes D_{2,1/8}(d)
\end{array}
\right]
\end{equation}
where $D_{2,1/8}(d)\in\mathcal{L}(X^\ast_{2,1/8})$ is an operator that depends in some suitable sense on the basic differentiation matrix $d\in\mathbb{R}^{3\times 3}$ of approximation order $2$ for the Gauss-Lobatto spectral elment method wich is defined by:
\begin{equation}
d:=
\left(
\begin{array}{ccc}
	-3/2& 2& -1/2\\
	-1/2&0&1/2\\
	1/2&-2&3/2
\end{array}
\right)
\end{equation}
now, since spectral methods of this kind have approximation order $m+2$ with respect to $\|\cdot\|_{X_{m,h}}$, we can take $\alpha=1/8$ and $\tau=(1/16)^3\norm*{\infty}{d^2}^{-1}$, using $H_{2,1/8}$ we can compute the basic element of the discrete semigroup $\{U_{3,k}:U_{3,k}=\hat{G}_\tau^k\}$ that will be defined in the form
\begin{equation}
\hat{G}_\tau:=1-i\tau H_{2,1/8}-\frac{(\tau H_{2,1/8})^{2}}{2}+\frac{i(\tau H_{2,1/8})^{3}}{6}
\end{equation}
using the elements of the discrete semigroup we can compute $\mathbf{\hat{X}}_{3,k}$ in the form:
\begin{equation}
\mathbf{\hat{X}}_{3,k}:=\hat{U}^\ast_{3,k}\mathbf{\hat{X}}\hat{U}_{3,k}
\end{equation}
and its expected value $\langle\mathbf{\hat{X}}_{3,k}\rangle$ can be computed using the expression:
\begin{equation}
\langle\mathbf{\hat{X}}_{3,k}\rangle:=(\mathbb{E}(x_{2,1/8}),\mathbb{E}(y_{2,1/8}))
\end{equation}
with
\begin{equation}
\mathbb{E}(\cdot)=\frac{1}{\mathcal{P}_{\psi}}\hat{\psi}^\ast_0\hat{U}_{3,k}^\ast\mathcal{M}_{2,1/8}(\cdot)\hat{U}_{3,k}\hat{\psi}_0
\end{equation}
and where
\begin{equation}
\mathcal{P}_\psi:=\hat{\psi}_0^\ast\hat{U}_{3,k}^\ast\mathcal{M}_{2,1/8}\hat{U}_{3,k}\hat{\psi}_0.
\end{equation}
now, since the particular factorization $\mathcal{H}_{2,1/8}:=\mathcal{M}_{2,1/8}H_{2,1/8}$ of $H\in\mathcal{L}(X)$ is clearly symmetric by T.\ref{tfp}  for the value of $\tau$ used in this example we can use theorems T.\ref{tscs} and T.\ref{tccs} presented above to predict the behavior of the evolution operators in the discrete semigroup relative to this quantum system. 
\end{example}

\medskip

\begin{center}
 \normalsize\scshape{Acknowledgements}
\end{center}

\medskip

I want to say thanks: To God... he does know why, to Mirna, my girl, for her love and support, for all that 
good moments everyday and for make me laugh, to my parents for their love and support, to Concepci{\'o}n Ferrufino and Rosibel Pacheco for their support and friendship, to Stanly Steinberg for his friendship and advice, to Rafael Ant{\'u}nez for his friendship and advice, to Eduardo Bravo for his advice and friendship, to Jorge Destephen for taking time to read the manuscript of many of the sections of this work, to C{\'e}sar, Sheila, Manolo, Daniel and Alex for their support in the first workshop in computational methods for partial differential equations that I organized. I am really grateful with them all.

\appendix

\section{Some Theorems from Linear Algebra}

In this section we will present some basic theorems from linear algebra that
are very useful in the study of some processes  and methods presented in this work.

\begin{theorem}
  Any symmetric positive/negative definite matrix is invertible.
\end{theorem}

\begin{theorem}
  For a given matrix $\mathbf{A} \in \mathbb{C}^{m \times n}$ we will have
  that
  \[ \left\| \mathbf{A} \right\|_{\infty} : = \max_i \sum_j \left|
     \mathbf{A}_{i, j} \right| . \]
\end{theorem}

\begin{theorem}
  \label{T11A}Every symmetric positive/negative definite matrix $\mathbf{A}
  \in \mathbb{C}^{m \times n}$ has all of its eigenvalues real
  positive/negative and its eigenvectors form an orthogonal system.
\end{theorem}

\begin{theorem}
  For a given matrix $\mathbf{A} \in \mathbb{C}^{m \times n}$ we will have
  that
  \[ \left\| \mathbf{A} \right\|_2 := \max \{ \sqrt{\lambda} : \lambda
     \in \sigma ( \mathbf{A}^{\ast} \mathbf{A})\}. \]
\end{theorem}

\begin{theorem}\label{tg1}
  Gershgorin Theorem. For any given matrix $\mathbf{A} \in \mathbb{C}^{m
  \times n}$ we will have that $\sigma ( \mathbf{A}) \subset \bigcup_i D_i$
  with
  \[ D_i := \{z \in \mathbb{C}: \left| z - \mathbf{A}_{i, i} \right|
     \leq \sum_{j \neq i} \left| \mathbf{A}_{i, j} \right| \}, 1 \leq i \leq
     m \]
     \end{theorem}

Using theorems T.\ref{tg1} and T.\ref{T11A} one can obtain the following.
  \begin{corollary}
    \label{T2inf}For any symmetric positive/negative definite/semi-definite matrix $\mathbf{A} \in
    \mathbb{C}^{m \times n}$ we will have that
    \[ \left\| \mathbf{A} \right\|_2 \leq \left\| \mathbf{A}
       \right\|_{\infty} . \]
  \end{corollary}
  
  \begin{theorem}
  \label{T15A}For any given matrix $\mathbf{A} \in \mathbb{C}^{n \times n}$
  and any $\lambda \in \sigma ( \mathbf{A})$ whose corresponding eigenvector
  is given by $v_{\lambda} \in \mathbb{C}^n$ we will have for every
  polynomial $p_m (z) := a_0 + a_1 z + \cdots a_m z^m \in \mathcal{P}_m
  (\mathbb{C})$, with $\mathcal{P}_m (\mathbb{C})$ the set of all
  polynomials of degree $\leq m$, that $p_m (\lambda) \in \sigma (p_m (
  \mathbf{A}))$ will be an eigenvalue of $p_m ( \mathbf{A}) \in
  \mathbb{C}^{n \times n}$ with corresponding eigenvector $v_{\lambda}\in\mathbb{C}^n$.
\end{theorem}

\section{A Theorem from Real Analysis.}

In this section we will present a basic theorem from real analysis that is
very useful in the study of some processes presented in this article.

\begin{theorem}
  \label{T16B}Let $\sum a_n$ be a number series such that
  \[ \lim_{n \rightarrow \infty} a_n = 0 \]
  such that the terms $a_n$ are alternating positive and negative and such
  that $\left| a_{n + 1} \right| \leq \left| a_n \right|$ for $n \geq 0$. Then
  the series converge and
  \begin{equation}
    \left| \sum_{n = 0}^{\infty} a_n \right| \leq \left| a_0 \right| .
  \end{equation}
\end{theorem}

\begin{remark}
  From the last theorem we also have
  \begin{equation}
    \left| \sum_{n = k}^{\infty} a_n \right| \leq \left| a_k \right|
  \end{equation}
  and
  \begin{equation}
    \left| \sum_{n = k}^m a_n \right| \leq \left| a_k \right| .
  \end{equation}
\end{remark}

\end{document}